\documentclass[10pt]{article}

\usepackage{amsmath,amsfonts,amssymb,amsthm}
\usepackage{color}
\usepackage{colortbl}
\usepackage{array}
\usepackage{mathrsfs}
\usepackage{dcolumn}
\usepackage{longtable}
\usepackage{hhline}
\usepackage{bm}

\newcommand{\bracketed}[2]{\genfrac{[}{]}{0pt}{0}{#1}{#2}}
\newcommand{\braced}[2]{\genfrac{\{}{\}}{0pt}{0}{#1}{#2}}

\newtheorem{thm}{Theorem}[section]
\newtheorem{defn}[thm]{Definition}

\newtheorem{rema}[thm]{Remark}

\def\ds{\displaystyle}

\title{\bf A $q$-Analogue of $r$-Whitney Numbers of the\\ Second Kind and Its Hankel Transform}
\author{{\large\bf Roberto B. Corcino}\\{\large\bf Jay M. Ontolan}\\ {\large\bf Jennifer Ca\~{n}ete}\\ 
Research Institute for Computational Mathematics and Physics\\
Cebu Normal University\\
Cebu City, Philippines 6000\\\\
{\large\bf Mary Joy R. Latayada}\\
Department of Mathematics\\
 Caraga State University\\
  Butuan City, Philippines 8600 
}

\date{}

\begin{document}

\maketitle

\begin{abstract}
	A $q$-analogue of $r$-Whitney numbers of the second kind, denoted by $W_{m,r}[n,k]_q$, is defined by means of a triangular recurrence relation. In this paper, several fundamental properties for the $q$-analogue are established including other forms of recurrence relations, explicit formulas and generating functions. Moreover, a kind of Hankel transform for $W_{m,r}[n,k]_q$ is obtained. 

\end{abstract}

\section{Introduction}
Several mathematicians developed a way of obtaining a generalization of some special numbers. One generalization is a $q$-analogue of these special numbers. A $q$-analogue is a mathematical expression parameterized by a quantity $q$ that generalizes a known expresssion and reduces to the known expression in the limit, as $q \rightarrow 1$. For instance, a polynomial $\ds a_k(q)$ is a $q$-analogue of the polynomial $a_k$ if $\lim\limits_{q\rightarrow 1}a_k(q)=a_k$. The $q$-analogue of $n$, $n!$, $(n)_k$ and $\dbinom{n}{k}$ are respectively given by
\begin{align*}
[n]_q &= 1+q+q^2+ \cdots +q^{n-1} = \frac{1-q^n}{1-q};\\
[n]_q! &= [n]_q[n-1]_q\cdots[2]_q[1]_q;\\
[n]_{k,q} &=[n]_q[n-1]_q\cdots[n-k+1]_q;\\
\begin{bmatrix}n\\k\end{bmatrix}_q &= \frac{[n]_q!}{[k]_q![n-k]_q!}=\frac{[n]_{k,q}}{k!}.
\end{align*}
The polynomials $\bracketed{n}{k}$ are usually called the {\it $q$-binomial coefficients}. These polynomials possess several properties including the $q$-binomial inversion formula
\begin{equation}\label{e1}
f_n=\sum\limits_{k=0}^{n}\begin{bmatrix}n\\k\end{bmatrix}_qg_k\Longleftrightarrow g_n=\sum\limits_{k=0}^{n}(-1)^{n-k}q^{\binom{n-k}{2}}\begin{bmatrix}n\\k\end{bmatrix}_qf_k.
\end{equation}
and certain generating function
\begin{equation*}
\sum\limits_{k=0}^{n} q^{\binom{k}{2}} \begin{bmatrix}n\\k\end{bmatrix}_q x^k = (1+x)(1+xq)(1+xq^2)\cdots (1+xq^{n-1}).
\end{equation*}
The $q$-binomial inversion formula is a $q$-analogue of the binomial inversion formula \cite{Com}
\begin{equation}\label{inv1}
f_n=\sum\limits_{k=0}^{n}\binom{n}{k}g_k\Longleftrightarrow g_n=\sum\limits_{k=0}^{n}(-1)^{n-k}\binom{n}{k}f_k.
\end{equation}

\smallskip
A $q$-analogue of both kinds of Stirling numbers was first defined by Carlitz in \cite{Car1}. The second kind of which, known as $q$-Stirling numbers of the second kind, is defined in terms of the following recurrence relation
\begin{equation}
S_q[n,k]=S_q[n-1,k-1]+[k]_qS_q[n-1,k]
\end{equation}
in connection with a problem in Abelian groups, such that when $q\to1$, this gives the triangular recurrence relation for the classical Stirling numbers of the second kind $S(n,k)$
\begin{equation}
S(n,k)=S(n-1,k-1)+kS(n-1,k).
\end{equation}
A different way of defining $q$-analogue of Stirling numbers of the second kind has been adapted in the paper by \cite{Ehr} which is given as follows
\begin{equation}\label{qSNEhr}
S_q[n,k]=q^{k-1}S_q[n-1,k-1]+[k]_qS_q[n-1,k].
\end{equation}
This type of $q$-analogue gives the Hankel transform of $q$-exponential polynomials and numbers which are certain $q$-analogue of Bell polynomials and numbers. 

\smallskip
Using the method of Aigner \cite{Aig} and Mezo \cite{Mezo3}, R. Corcino and C. Corcino \cite{Cor1} have successfully established the Hankel transform $H(G_{n,r,\beta})$ of the sequence of generalized Bell numbers which is given by
$$H(G_{n,r,\beta})=\prod_{j=0}^n\beta^jj!$$ 
where $G_{n,r,\beta}$ is the sum of $(r,\beta)$-Stirling numbers $\braced{n}{k}_{r,\beta}$, that is, 
\begin{equation*}
G_{n,r,\beta}=\sum_{k=0}^n\braced{n}{k}_{r,\beta}
\end{equation*}
(see \cite{Cor2, Cor3}). These numbers are also known as $(r,\beta)$-Bell numbers. In the same paper \cite{Cor1}, the authors have tried to establish the Hankel transform for the $q$-analogue of Rucinski-Voigt numbers \cite{Cor4}, which are also equivalent to $(r,\beta)$-Bell numbers.  However, the authors were not successful in their attempt. Recently, by introducing a new way of defining the $q$-analogue of Stirling-type and Bell-type numbers, R. Corcino et al.\cite{Cor0} were able to establish the Hankel transform for the $q$-analogue of non-central Bell numbers. One can easily verify that non-central Bell numbers are special case of $(r,\beta)$-Bell numbers (or Rucinski-Voigt numbers or $r$-Dowling numbers). In the desire to establish the Hankel transform of $q$-analogue of $r$-Dowling numbers, the present authors have made the initiative to define and investigate a $q$-analogue of $r$-Whitney numbers of the second kind parallel to the $q$-analogues of noncentral Stirling numbers of the second kind in \cite{Cor0}. It is worth-mentioning that $r$-Whitney numbers of the second kind and  $r$-Dowling numbers are equivalent to $(r,\beta)$-Stirling numbers and $(r,\beta)$-Bell numbers, respectively.

\bigskip
\section{A $q$-Analogue of $W_{m,r}(n,k)$ and Recurrence Relations}
Now, let us introduce the desired definition of the $q$-analogue of $r$-Whitney numbers of the second kind.
\begin{defn}
For non-negative integers $n$ and $k$, a q-analogue  $W_{m,r}[n,k]_{q}$ of  $W_{m,r}(n,k)$ is defined by
\smallskip
 \begin{equation}
 W_{m,r}[n,k]_{q}= q^{m(k-1)+r}W_{m,r}[n-1,k-1]_q+[mk+r]_{q} W_{m,r}[n-1,k]_{q}. \label{j}
\end{equation}
\end{defn}
\noindent where $W_{m,r}[0,0]_{q}=1, W_{m,r}[n,k]_{q}=0$ for $n<k$ or $n,k < 0 $ and $[t-k]_{q}=\frac{1}{q^{k}}([t]_{q}-[k]_{q})$.

\smallskip
\begin{rema} When $q \rightarrow 1 $, the above definition will reduce to the recurrence relation of the $r$-Whitney numbers of the second kind established by Mezo which is given by
\begin{equation*}
	W_{m,r}(n,k) = W_{m,r}(n-1,k-1) + (mk+r)W_{m,r}(n-1,k).
\end{equation*}	
\end{rema}

\smallskip
\begin{rema} It can easily be verified that
 \begin{equation*}
	W_{m,r}[n,0]=[r]_{q}^{n}.
	\end{equation*}
\end{rema}

By proper application of \eqref{j}, we can easily obtain two other forms of recurrence relations and certain generating function.

\begin{thm} For nonnegative integers $n$ and $k$, the $q$-analogue $W_{m,r}[n,k]_{q}$ satisfies the following vertical and horizontal recurrence relations:
	\begin{equation}\label{vertrr}
	W_{m,r}[n+1,k+1]_{q}=q^{mk+r}\sum_{j=k}^{n}[m(k+1)+r]_q^{n-j}W_{m,r}[j,k]_{q}
	\end{equation}
	
	\begin{equation}\label{horizrr}
	W_{m,r}[n,k]_{q}=\sum_{j=0}^{n-k}(-1)^{j}q^{-r-m(k+j)}\frac{{r}_{{k+j+1},q}}{{r}_{{k+1},q}}W_{m,r}[n+1,k+j+1]_{q}
	\end{equation}
respectively, where\\
$${r}_{i,q}=\prod_{h=1}^{i-1}q^{-r-mh+m}[mh+r]_{q}$$
with initial value $W_{m,r}[0,0]_{q}=1.$
	\end{thm}

\begin{proof}
	Replacing $n$ by $n+1$ and $k$ by $k+1$ in (\ref{j}) gives 
	\begin{equation}
	W_{m,r}[n+1,k+1]_{q}=q^{m(k)+r}W_{m,r}[n,k]_{q}+[m(k+1)+r]_{q}W_{m,r}[n,k+1]
	\end{equation}
	Applying repeatedly by (\ref{j}) gives
	\begin{eqnarray}
	W_{m,r}[n+1,k+1]_{q}&=&q^{m(k)+r}W_{m,r}[n,k]_{q}+[m(k+1)+r]_{q}\nonumber\\
	&&(q^{mk+r} W_{m,r}[n-1,k]_{q}+ [m(k+1)+r]_{q}W_{m,r}[n-1,k]_{q})\nonumber\\
	&=&q^{m(k)+r}W_{m,r}[n,k]_{q}+[m(k+1)+r]_{q} W_{m,r}[n-1,k]_{q} +[m(k+1)+r]^2\nonumber \\
	&&(q^{mk+r} W_{m,r}[n-2,k]_{q}+[m(k+1)+r]q W_{m,r}[n-2,k])\nonumber\\
	&&+\cdots + [m(k+1)+r]_{q}^{n-k}( q^{mk+r} W_{m,r}[k+1,k+1]_{q})\nonumber
\end{eqnarray}
Using the fact that $W_{m,r}[k+1,k+1]_{q}=W_{m,r}[k,k]_{q}$, we obtain \eqref{vertrr}.  Now, to prove \eqref{horizrr}, we have to rewrite the right-hand side (RHS) of the relation using \eqref{j} as follows
\small\begin{align*}
  RHS&=\sum_{j=0}^{n-k}(-1)^{j}q^{-r-m(k+j)}\frac{{r}_{{k+j+1},q}}{{r}_{{k+1},q}}\\
 &\left(q^{m(k+j)+r}W_{m,r}[n,k+j]_{q}+[m(k+j+1)+r]_{q}W_{m,r}[n,k+j+1]_{q}\right)\\
 &=\sum_{j=0}^{n-k}(-1)^{j}q^{-r-m(k+j)}\frac{{r}_{{k+j+1},q}}{{r}_{{k+1},q}}{q^{m(k+j)+r}W_{m,r}[n,k+j]_{q}}\\
 &+\sum_{j=0}^{n-k}(-1)^{j}q^{-r-m(k+j)}\frac{{r}_{{k+j+1},q}}{{r}_{{k+1},q}}[m(k+j+1)+r]_{q}W_{m,r}[n,k+j+1]_{q}\\
 &=\sum_{j=0}^{n-k}(-1)^{j}\frac{{r}_{{k+j+1},q}}{{r}_{{k+1},q}}W_{m,r}[n,k+j]_{q}\\
 &+\sum_{j=1}^{n-k}(-1)^{j-1}q^{-r-m(k+j-1)}\frac{\prod_{h=1}^{k+j-1}q^{-r-mh+m}{[mh+r]_{q}}}{\prod_{h=1}^{k}q^{-r-mh+m}{[mh+r]_{q}}}[m(k+j)+r]_{q}W_{m,r}[n,k+j]_{q}\\
\end{align*}
Note that
\begin{eqnarray}
\prod_{h=1}^{k+j-1}q^{-r-mh+m}{[mh+r]}_{q} &=(q^{-r-m+m}[m+r]_{q})(q^{-r-m(2)+m}[m(2)+r]_{q})\nonumber\\
&\cdots(q^{-r-m(k+j-1)+m}[m(k+j-1)+r]_{q}).\nonumber
\end{eqnarray}
Now, we have
\begin{align*}
 RHS&=\sum_{j=0}^{n-k}(-1)^{j}\frac{{r}_{{k+j+1},q}}{{r}_{{k+1},q}}W_{m,r}[n,k+j]_{q}\\
 &+\frac{1}{\prod_{h=1}^{k}q^{-r-mh+m}{[mh+r]_{q}}}\sum_{j=1}^{n-k}(-1)^{j-1}q^{-r-m(k+j-1)}(q^{-r-m+m}[m+r]_{q})\\
 &(q^{-r-m(2)+m}{[m(2)+r]_{q}})\cdots(q^{-r-m(k+j-1)+m}{[m(k+j-1)+r]_{q}})[m(k+j)+r]_{q}W_{m,r}[n,k+j]_{q}\\
 &=\sum_{j=0}^{n-k}(-1)^{j}\frac{{r}_{{k+j+1},q}}{{r}_{{k+1},q}}W_{m,r}[n,k+j]_{q}+\sum_{j=0}^{n-k}(-1)^{j-1}\frac{{r}_{{k+j+1},q}}{{r}_{{k+1},q}}W_{m,r}[n,k+j]_{q}\\
 &=W_{m,r}[n,k]_{q}
\end{align*}
This proves the theorem.
\end{proof}

\begin{thm}
A horizontal generating function of $W_m,r[n,k]_{q}$ is given by
\begin{equation}
\sum_{k=0}^{n}W_{m,r}[n,k]_{q}[t-r|m]_{k,q}=[t]_q^{n}.\label{o}
\end{equation}
\end{thm}
\begin{proof}
We will prove this by induction on $n$. It is easy to verify that $\eqref{o}$ holds when $n=0$.
$$W_{m,r}[0,0][t-r|m]_{0,q}=1=[t]_{q}^0$$\\
Now, suppose that it is true for some $n\geq0$. That is,
\begin{equation}
\sum_{k=0}^{n}W_{m,r}[n,k]_{q}[t-r|m]_{k,q}=[t]_q^{n}.
\end{equation}
Then
\small\begin{align*}
&\sum_{k=0}^{n+1}W_{m,r}[n+1,k]_{q}[t-r|m]_{k,q}\\
&\;\;\;\;\;\;\;\;\;\;=\sum_{k=0}^{n+1}\left\{\left\{q^{m(k-1)+r} W_{m,r}[n,k-1]_{q}+[mk+r]_{q}W_{m,r}[n,k]_{q}\right\}[t-r|m]_{k,q}
\right\}\\
&\;\;\;\;\;\;\;\;\;\;=\sum_{k=0}^{n+1}q^{m(k-1)+r} W_{m,r}[n,k-1]_{q}[t-r|m]_{k,q}+\sum_{k=0}^{n+1}[mk+r]_{q}W_{m,r}[n,k]_{q}[t-r|m]_{k,q}\\
&\;\;\;\;\;\;\;\;\;\;=\sum_{k=1}^{n}q^{m(k)+r} W_{m,r}[n,k]_{q}[t-r|m]_{k+1,q}+\sum_{k=0}^{n}[mk+r]_{q}W_{m,r}[n,k]_{q}[t-r|m]_{k,q}\\
&\;\;\;\;\;\;\;\;\;\;=\sum_{k=0}^{n}q^{m(k)+r} W_{m,r}[n,k]_{q}[t-r|m]_{k,q}[t-r-km]_{q}+\sum_{k=0}^{n}[mk+r]_{q}W_{m,r}[n,k][t-r|m]_{k,q}.\\
\end{align*}
Since $[t-k]_{q}=\frac{1}{q^{k}}([t]_{q}-[k]_{q})$, letting $k=r+km$, we have,
\begin{equation*}
[t-r-km]_{q}=q^{-(km+r)}([t]_{q}-[km+r]_{q}).
\end{equation*}
Finally,
\begin{align*}
&\sum_{k=0}^{n+1}W_{m,r}[n+1,k]_{q}[t-r|m]_{k,q}\\
&\;\;\;\;\;\;\;\;\;\;=\sum_{k=0}^{n}q^{m(k)+r} W_{m,r}[n,k]_{q}[t-r|m]_{k,q}q^{-(m(k)+r)}([t]_{q}-[km+r]_{q})\\
&\;\;\;\;\;\;\;\;\;\;+\sum_{k=0}^{n}[mk+r]_{q}W_{m,r}[n,k]_{q}[t-r|m]_{k,q}\\
&\;\;\;\;\;\;\;\;\;\;=\sum_{k=0}^{n}[t]_{q}W_{m,r}[n,k]_{q}[t-r|m]_{k,q}\\
&\;\;\;\;\;\;\;\;\;\;=[t]_{q}\sum_{k=0}^{n}W_{m,r}[n,k]_{q}[t-r|m]_{k,q}=[t]_{q}[t]_{q}^{n}=[t]_{q}^{n+1}
\end{align*}
\rm This proves the theorem.
\end{proof}

\section{Explicit Formula and Generating Function}	
Explicit formulas and generating functions of a given sequence of numbers or polynomials are useful tools in giving combinatorial interpretation of the numbers or polynomials. In the subsequent theorems, we establish the exponential and rational generating functions and explicit formulas for $W_{m,r}[n,k]_{q}$.\\

A $q$-analogue of the difference operator, denoted by $\Delta_{q,h}^{n}$, also known as $q$-difference operator of order $n$, was defined by the rule
\begin{equation}
\Delta_{q,h}^{n} f(x) = \left[\prod_{j=0}^{n-1} (E_h-q^{j})\right] f(x),\quad n\geq 1,
\end{equation}
where $E_h$ is the shift operator defined by $E_{h} f(x) = f(x+h)$. When $h =1$, we use the notation
\begin{equation}
\Delta_{q,h}^{n} = \Delta_{q}^{n}
\end{equation}
This operator was thoroughly discussed in [9][19]. By convention, it is defined that $\Delta_{q,h}^{0} = 1 $ (identity map). The following is the explicit formula for the $q$-difference operator
\begin{equation}
\Delta_{q,h}^{k} f(x)= \sum_{j=0}^{k}(-1)^{k-j}q^{\binom{k-j}{2}}{\bracketed{k}{j}}_{q}f(x+jh).\label{do}
\end{equation}

The new $q$-analogue of Newton's Interpolation Formula in [19] states that, for
\begin{equation*}
f_{q}(x) = a_{0}+a_{1}[x-x_{0}]_{q}+\cdots+a_{k}[x-x_0]_{q}[x-x_{1}]_{q}\cdots[x-x_{k-1}]_{q},
\end{equation*}
we have
\begin{align*}
f_{q}(x)&=f_{q}(x_0)+\frac{\Delta_{{q^h},h}f_{q}(x_0)[x-x_{0}]_{q}}{[1]_{q^h}![h]_{q}}+\frac{\Delta_{{q^h},h}^{2}f_{q}(x_0)[x-x_{0}]_{q}[x-x_{1}]_{q}}{[2]_{q^h}![h]_{q}^{2}}\\
&+\cdots+\frac{\Delta_{{q^h},h}^{k}f_{q}(x_0)[x-x_{0}]_{q}[x-x_{1}]_{q}\cdots[x-x_{k-1}]_{q}}{[k]_{q^h}![h]_{q}^{k}}
\end{align*}
where $x_{k}=x_{0}+kh$, $k=1,2,\cdots$ such that when $x_{0}=0$ and $h=m$ this can be simplified as \\
\begin{align*}
f_{q}(x)&=f_{q}(0)+\frac{\Delta_{{q^m},m}f_{q}(0)[x]_{q}}{[1]_{q^m}![m]_{q}}+\frac{\Delta_{{q^m},m}^{2}f_{q}(0)[x]_{q}[x-m]_{q}}{[2]_{q^m}![m]_{q}^{2}}\\
&+\cdots+\frac{\Delta_{{q^m},m}^{k}f_{q}(0)[x]_{q}[x-m]_{q}\cdots[x-m(k-1)]_{q}}{[m]_{q^m}![m]_{q}^{k}}.
\end{align*}
\rm Using $\eqref{o}$ with $t=x$, we get 
\begin{equation}
\sum_{k=0}^{n} W_m,r[n,k]_{q}[x-r|m]_{k,q} = [x]_{q}^{n}
\end{equation}
\rm which can be expressed further as
\begin{equation}
\sum_{k=0}^{n}W_{m,r}[n,k][x]_{q}[x-m]_{m}[x-2m]_{q}\cdots[x-(k-1)m]=[x+r]_{q}^{n}
\end{equation}
 
Let $f_{q}(x)=[x+r]_{q}^{n}$ and $W_{m,r}[n,k]_{q} = \frac{\Delta_{q^{m},m}^{k} f(_{q}(0)}{[k]_{q^m}![m]_{q}^{k}}$. By proper application of the above Newton's Interpolation Formula and the identity in \eqref{do}, we get
\begin{align*}
\Delta_{q^{m},m}^{f} f_{q}(x)&=\sum_{j=0}^{k}(-1)^{k-j}q^{m\binom{k-j}{2}}{\bracketed{k}{j}}_{q^m}f_{q}(x+jm)\\
&=\sum_{j=0}^{k}(-1)^{k-j}q^{m\binom{k-j}{2}}{k\brack j}_{q^m}[x-r+jm]_{q}^{n}
\end{align*}
Evaluating at $x=0$ yields
\begin{equation*}
\Delta_{q^{m},m}^{f} f_{q}(0)=\sum_{j=0}^{k}(-1)^{k-j}q^{m\binom{k-j}{2}}{\bracketed{k}{j}}_{q^m}[r+jm]_{q}^{n},
\end{equation*}
which gives the following explicit formula for $W_{m,r}[n,k]_{q}$.
\begin{thm}\label{EFqWN}
	The explicit formula for $W_{m,r}[n,k]_{q}$ is given by \\
	\begin{equation}\label{exp-1}
	W_{m,r}[n,k]_{q}=\frac{1}{[k]_{q^m}![m]_{q}^{k}}\sum_{j=0}^{k} (-1)^{k-j} q^{m\binom{k-j}{2}}{\bracketed{k}{j}}_{q^m}[jm+r]_{q}^{n}.
	\end{equation}
\end{thm}
\begin{rema}\rm
	The above theorem reduces to
	\begin{equation*}
	W_{m,r}(n,k)=\frac{1}{k!{m}^{k}}\sum_{j=0}^{k} (-1)^{k-j}{\binom{k}{j}}(jm+r)^{n}
	\end{equation*}
	when $q\rightarrow 1$, which is exactly the explicit formula of the $r-$Whitney numbers of the second kind.
\end{rema}
\begin{rema}\rm
For brevity, $\eqref{exp-1}$ can be expressed as 
\begin{equation}
W_{m,r}[n,k]_{q}=\frac{1}{[k]_{q^m}![m]_{q}^{k}}\left[\Delta_{q^{m},m}^{k}[x+r]_{q}^{n}\right]_{x=0}
\end{equation}
\end{rema}

\begin{thm}
	For nonnegative integers $n$ and $k$, the $q$-analogue $W_{m,r}[n,k]_{q}$ has a generating function

	\begin{equation}
	\sum_{n\geq 0}W_{m,r}[n,k]_{q}\frac{[t]_{q}^{n}}{[n]_{q}!}
	=\frac{1}{[k]_{q}m![m]_{q}^{k}}\big[\Delta_{q^{m},m^k}e_q\big([x+jm+r]_{q}[t]_{q}\big)\big]_{x=0}.
	\end{equation}
\end{thm}

\begin{proof}
	Using the formula in Theorem \ref{EFqWN}, we obtain
\begin{align*}
	\sum_{n\geq 0}W_{m,r}[n,k]_{q}\frac{[t]_{q}^{n}}{[n]_{q}!}&=\sum_{n\geq0}\frac{1}{[k]_{q}m![m]_{q}^{k}}\sum_{j=0}^{k} (-1)^{k-j} q^{m\binom{k-j}{2}}{\bracketed{k}{j}}_{q^m}[jm+r]_{q}^{n}\frac{[t]_{q}^{n}}{[n]_{q}!}\\
	&=\sum_{j=0}^{k}\frac{1}{[k]_{q}m![m]_{q}^{k}}(-1)^{k-j}q^{m\binom{k-j}{2}}{\bracketed{k}{j}}_{q^m}\sum_{n\geq 0}\frac{\big([jm+r]_{q}[t]_{q}\big)^{n}}{[n]_{q}!}\\
	&=\frac{1}{[k]_{q}m![m]_{q}^{k}}\sum_{j=0}^{k}(-1)^{k-j} q^{m\binom{k-j}{2}}{\bracketed{k}{j}}_{q^m}e_{q}\big([jm+r]_{q}[t]_{q}\big)\\
	&=\frac{1}{[k]_{q}m![m]_{q}^{k}}\big[\Delta_{q^{m},m^k}e_q\big([x+jm+r]_{q}[t]_{q}\big)\big]_{x=0}.
\end{align*}
\end{proof}
\begin{rema}\rm
	When $q\rightarrow 1$, the above theorem becomes
	\begin{align*}
	\sum_{n\geq 0}W_{m,r}(n,k)_{q}\frac{t^{n}}{n!}&=\frac{1}{k!(m)^{k}}\sum_{j=0}^{k} (-1)^{k-j}{\binom{k}{j}} e^{(jm+r)t}\\
	&=\frac{e^{rt}}{k!(m)^{k}}\sum_{j=0}^{k}(-1)^{k-j}{\binom{k}{j}}\left(e^{mt}\right)^{j}\\
	&=\frac{e^{rt}}{k!(m)^{k}}\left(e^{mt}-1\right)^{k},
	\end{align*}
which is the exponential generating function of the $r$-Whitney numbers of the second kind.
\end{rema}

 \begin{thm} \label{rgf}
 For nonnegative integers $n$ and $k$, the $q$-analogue $W_m,r[n,k]_{q}$ satisfies the rational generating function 
\begin{equation}
\Psi_{k}(t)=\sum_{n\geq k}W_{m,r}[n,k]_{q}[t]_{q}^{n}=\frac{q^{{{m}\binom{k}{2}}+kr}[t]_{q}^{k}}{\prod_{j=0}^{k}(1-[mj+r]_{q}[t]_{q})}.
\end{equation}
\end{thm}
\begin{proof} It can easily be verified that, when $k=0$, we obtain
	\begin{equation*}
	\Psi_{0}(t)=\sum_{n\geq 0}W_{m,r}[n,0]_{q}[t]_{q}^{n}
	=\sum_{n\geq 0}[r]_{q}^{n}[t]_{q}^{n}
	=\sum_{n\geq 0}\left([r]_{q}[t]_{q}\right)^{n}
	=\frac{1}{1-[r]_q[t]_{q}}.
	\end{equation*}
Now, for $k\geq 0$ and using the recurrence relation in $\eqref{j}$, we obtain

\begin{align*}
\Psi_{k}(t)&=\sum_{n\geq k}W_{m,r}[n,k]_{q}[t]_{q}^{n}=\sum_{n\geq k}\left({q^{m(k-1)+r}} W_{m,r}[n-1,k-1]_{q}\right.\\
&\;\;\;\left.+[mk+r]_{q}W_{m,r}[n-1,k]_{q}\right)[t]_{q}^{n}\\
&=\sum_{n\geq k}{q^{m(k-1)+r}} W_{m,r}[n-1,k-1]_{q}[t]_{q}^{n}+\sum_{n\geq k}[mk+r]_{q}W_{m,r}[n-1,k]_{q}[t]_{q}^{n}\\
&={q^{m(k-1)+r}}\sum_{n\geq k}W_{m,r}[n-1,k-1]_{q}[t]_{q}^{n}+[mk+r]_{q}\sum_{n\geq k}W_{m,r}[n-1,k]_{q}[t]_{q}^{n}\\
&={q^{m(k-1)+r}}[t]_{q}\sum_{n\geq k}W_{m,r}[n-1,k-1]_{q}[t]_{q}^{n-1}+[mk+r]_{q}[t]_{q}\sum_{n\geq k}W_{m,r}[n-1,k]_{q}[t]_{q}^{n-1}.
\end{align*}
This can be expressed as 
\begin{align*}
\Psi_{k}(t)&={q^{m(k-1)+r}}[t]_{q}\Psi_{k-1}(t)+[mk+r]_{q}[t]_{q}\Psi_{k}(t)\\
\Psi_{k}(t)-[mk+r]_{q}[t]_{q}\Psi_{k}(t)&={q^{m(k-1)+r}}[t]_{q}\Psi_{k-1}(t)\\
\Psi_{k}(t)(1-[mk+r]_{q}[t]_{q})&={q^{m(k-1)+r}}[t]_{q}\Psi_{k-1}(t).
\end{align*}
This gives
$$\Psi_{k}(t)=\frac{q^{m(k-1)+r}[t]_{q}}{1-[mk+r]_{q}[t]_{q}}\Psi_{k-1}(t).$$
By backward substitution, we get
\begin{align*}
\Psi_{k}(t)&=\frac{q^{m(k-1)+r}[t]_{q}}{1-[mk+r]_{q}[t]_{q}}\cdot \frac{q^{m(k-2)+r}[t]_{q}}{1-[m(k-1)+r]_{q}[t]_{q}}\Psi_{k-2}(t)\\	
&=\frac{q^{m(k-1)+r}[t]_{q}}{1-[mk+r]_{q}[t]_{q}}\cdot \frac{q^{m(k-2)+r}[t]_{q}}{1-[m(k-1)+r]_{q}[t]_{q}}
\cdot \frac{q^{m(k-3)+r}[t]_{q}}{1-[mk-2+r]_{q}[t]_{q}}\Psi_{k-3}(t)\\
				\vdots\\
\Psi_{k}(t)&=\frac{q^{m{\binom{k}{2}}+kr}[t]_{q}^{k}}{\prod_{j=0}^{k}\big(1-[mj+r]_{q}[t]_{q}\big)}.
\end{align*}
\end{proof}

\smallskip
As a consequence of Theorem \ref{rgf}, we have the following explicit formula in symmetric function form.

\begin{thm}
 For nonnegative integers $n$ and $k$, the explicit formula for $W_{m,r}[n,k]_{q}$ in the homogeneous symmetric function form is given by
 \begin{equation}\label{symm}
 W_{m,r}[n,k]_{q}=\sum_{0 \leq j_1\leq j_2\leq \cdots \leq j_{n-k}\leq k}q^{m{\binom{k}{2}}+kr}\;\;{\prod_{i=1}^{n-k}[mj_i+r]_{q}}.
\end{equation}
\end{thm}

\begin{proof} The rational generating function in Theorem \ref{rgf} can be expressed as
\begin{align*}
\sum_{n\geq k}W_{m,r}[n,k]_{q}[t]_{q}^{n}&=q^{m{\binom{k}{2}}+kr}[t]_{q}^{k}\;\;\prod_{j=0}^{k}\frac{1}{\left(1-[mj+r]_{q}[t]_{q}\right)}\\
&=q^{m{\binom{k}{2}}+kr}[t]_{q}^{k}\prod_{j=0}^{k}\sum_{n\geq 0}[mj+r]_{q}^{n}[t]_{q}^{n}\\
&=q^{m{\binom{k}{2}}+kr}[t]_{q}^{k}\cdot\sum_{n\geq k} {\sum_{S_1+S_2+\cdots S_k=n-k}}\;\;{\prod_{j=0}^{k}[mj+r]_{q}^{S_j}[t]_{q}^{S_j}}\\
&=q^{m{\binom{k}{2}}+kr}[t]_{q}^{k}\sum_{n\geq k}{{\sum_{S_1+S_2+\cdots S_k=n-k}}\;\;\prod_{j=0}^{k}[mj+r]_{q}^{S_j}}[t]_{q}^{S_j}\\
&=q^{m{\binom{k}{2}}+kr}[t]_{q}^{k}\sum_{n\geq k} {\sum_{S_1+S_2+\cdots S_k=n-k}}\;\;\prod_{j=0}^{k}[mj+r]_{q}^{S_j}\}[t]_{q}^{n-k}\\
&=q^{m{\binom{k}{2}}+kr}\sum_{n\geq k}{\sum_{S_1+S_2+\cdots S_k=n-k}}\;\;\prod_{j=0}^{k}[mj+r]_{q}^{S_j}[t]_{q}^{n}.
\end{align*}
Thus, by comparing the coefficients of $[t]_{q}^{n}$, we obtain
\begin{equation*}
W_{m,r}[n,k]_{q}=q^{m{\binom{k}{2}}+kr}{\sum_{S_1+S_2+\cdots S_k=n-k}}\;\;{\prod_{j=0}^{k}[mj+r]_{q}^{S_j}},
\end{equation*}
which is equivalent to the desired explicit formula in \eqref{symm}.
\end{proof}

\section{The Hankel Transform}
Now, we define another form of $q$-analogue of $r$-Whitney numbers of the second kind, denoted by $W^*_{m,r}[n,k]_{q}$, as 
$$W^*_{m,r}[n,k]_{q}=q^{-m{\binom{k}{2}}-kr}W_{m,r}[n,k]_{q}.$$
Then, using equation \eqref{symm}, we have
\begin{equation}\label{symm1}
W^*_{m,r}[n,k]_{q}=\sum_{0 \leq j_1\leq j_2\leq \cdots \leq j_{n-k}\leq k}{\prod_{i=1}^{n-k}[mj_i+r]_{q}}.
\end{equation}
The complete symmetric function of degree $n$ in $k$ variables $x_1, x_2, \ldots, x_k$, denoted by $h_n(x_1, x_2, \ldots, x_k)$, is defined by
\begin{equation}\label{compsymm1}
h_n(x_1, x_2, \ldots, x_k)=\sum_{1\leq j_1\leq j_2\leq \cdots \leq j_{n}\leq k}\;\;{\prod_{i=1}^{n}x_{j_i}},
\end{equation}
with initial condition $h_0(x_1, x_2, \ldots, x_k)=1$ and $h_{n-k}(x_1, x_2, \ldots, x_k)=0$ if $n<k$.  

\smallskip
Consider the case where  $x_i=[mi+r]_{q}$ for $i=0, 1, \ldots, k$.  Hence, we can express the other form of $q$-analogue of $r$-Whitney numbers of the second kind as
$$W^*_{m,r}[n+k,k]_{q}=h_n(x_0, x_1, \ldots, x_k).$$
Moreover, it can easily be shown that
\begin{equation}\label{eq1}
W^*_{m,r+mk}[s,t-k]_{q}=h_{s-t+k}(x_k, x_{k+1}, \ldots, x_t).
\end{equation}
Note that
$$h_n(x_k)=\sum_{k\leq j_1\leq j_2\leq \cdots \leq j_{n}\leq k}{\prod_{i=1}^{n}x_{j_i}}=x_k^n.$$
So, when $t=k$, we have
\begin{equation}\label{initial1}
W^*_{m,r+mk}[s,0]_{q}=h_{s}(x_k)=x_k^s=[mk+r]_{q}^s.
\end{equation}

\smallskip
An {\it $A$-tableau} is a list $\phi$ of column $c$ of a Ferrer's diagram of a partition $\lambda$(by decreasing order of length) such that the lengths $|c|$ are part of the sequence $A=(r_i)_{i\ge0}$, a strictly increasing sequence of nonnegative integers. Let $\omega$ be a function from the set of nonnegative integers $N$ to a ring K (column weights according to length). Suppose $\Phi$ is an $A$-tableau with $l$ columns of lengths $|c|\le h$. We use $T_r^A(h,l)$ to denote the set of such $A$-tableaux. Then, we set 
\begin{equation*}
\omega_A(\Phi)=\prod_{c\in\Phi}\omega(|c|).
\end{equation*}
Note that $\Phi$ might contain a finite number of columns whose lengths are zero since $0\in A=\{0,1,2,\ldots,k\}$ and if $\omega(0)\neq0$.

From this point onward, whenever an $A$-tableau is mentioned, it is always associated with the sequence $A=\{0,1,2,\ldots,k\}$.\\

Consider $\omega(|c|)=[m|c|+r]_q$ where $r$ is a complex number, and $|c|$ is the length of column $l$ of an $A$-tableau in $T_r^A(k,n-k)$. Then
\begin{eqnarray*}
W^*_{m,r}[n,k]=\sum_{\phi\in T_r^A(k,n-k)}\prod_{c\in\phi}\omega(|c|).
\end{eqnarray*}

\smallskip
Suppose
\begin{eqnarray*}
&\phi_1&\mbox{is a tableau with}\;k-l\;\mbox{columns whose lengths are in the}\;\mbox{set}\\
&&\{0,1,\ldots,l\},\;\mbox{and}\\
&\phi_2&\mbox{be a tableau with}\;n-k-j\;\mbox{columns whose lengths are in the}\\
&&\;\mbox{set}\;\{l+1,l+2,\ldots,l+j+1\}
\end{eqnarray*}
\noindent Then $$\phi_1\in T^{A_1}(l,k-l) \mbox{ and } \phi_2\in T^{A_2}(j,n-k-j)$$ where $A_1=\{0,1,\ldots,l\}$ and $A_2=\{l+1,l+2,\ldots,l+j+1\}$.  Notice that by joining the columns of $\phi_1$ and $\phi_2$, we obtain an $A$-tableau $\phi$ with $n-l-j$ columns whose lengths are in the set $A=A_1\cup A_2=\{0,1,\ldots,l+j+1\}$.  That is, $\phi\in T^A(l+j+1,n-l-j)$.  Then,
\begin{equation*}
\sum_{\phi\in T^A(l+j+1,n-l-j)}\omega_A(\phi)=\sum^{n-j}_{k=l}\left\{\sum_{\phi_1\in T^{A_1}(l,\;k-l)}\omega_{A_1}(\phi_1)\right\}\left\{\sum_{\phi_2\in T^{A_2}(j,\;n-k-j)}\omega_{A_2}(\phi_2)\right\}.
\end{equation*}
\noindent Note that
\begin{eqnarray*}
\sum_{\phi_2\in T^{A_2}(j,\;n-k-j)}\omega_{A_2}(\phi_2)&=&\sum_{\phi_2\in T^{A_2}(j,\;n-k-j)}\prod_{c\in\phi_2}[m|c|+r]_q\\
&=&\sum_{l+1\le g_1\le\ldots\le g_{\;n-k-j}\le\;l+j+1}\prod^{n-k-j}_{i=1}[mg_i+r]_q\\
&=&\sum_{0\le g_1\le\ldots\le g_{\;n-k-j}\le j}\prod^{n-k-j}_{i=1}[mg_i+m(l+1)+r]_q.
\end{eqnarray*}

Thus,                                                           
\begin{eqnarray*}
\sum_{0\le g_1\le\ldots\le g_{n-l-j}\le l+j+1}\prod^{n-l-j}_{i=1}[mg_i+r]_q\qquad\qquad\qquad\qquad\qquad\qquad\qquad\qquad
\end{eqnarray*}

\begin{eqnarray*}
&=&\!\!\!\sum^{n-j}_{k=l}\!\!\left\{\sum_{0\le g_1\le\ldots\le g_{k-l}\le l}\prod^{k-l}_{i=1}\![mg_i+r]_q\!\right\}\!\!\left\{\sum_{0\le g_1\le\ldots\le g_{n\!-\!k\!-\!j}\le j}\!\!\prod^{n-k-j}_{i=1}\!\!\![mg_i+m(l+1)+r]_q\!\right\}\!.
\end{eqnarray*}
Then the $q$-analogue $W^*_{m,r}[n,k]$ satisfies the following convolution-type identity 
\begin{equation*}
W^*_{m,r}[n+1,l+j+1]_q=\sum^n_{k=0}W^*_{m,r}[k,l]_qW^*_{m,r+m(l+1)}[n-k,j]_q.
\end{equation*}

\smallskip
Using the same argument above with
\begin{eqnarray*}
&\phi_1&\;\mbox{be a tableau with}\;l-k\;\mbox{columns whose lengths are in}\\
&&\;A_1=\{0,1,\ldots,k\},\;\mbox{and}\\
&\phi_2&\;\mbox{be a tableau with $j-n+k$ columns whose lengths are in}\\
&&\;A_2=\{k,k+1,\ldots,n\}.
\end{eqnarray*}                                                                             
That is, $\phi_1\in T^{A_1}(k,l-k)$ and $\phi_2\in T^{A_2}(n-k,j-n+k)$. We can easily obtain the following convolution formula:
\begin{equation*}
W^*_{m,r}[l+j,n]_q=\sum^{n-j}_{k=l}W^*_{m,r}[l,k]_qW^*_{m,r+mk}[j,n-k]_q.
\end{equation*}
This can further be written as
$$W^*_{m,r}[s+p,t]_{q}=\sum_{k=max\{0,t-p\}}^{min\{t,s\}}W^*_{m,r}[s,k]_{q}W^*_{m,r+mk}[p,t-k]_{q}.$$
Replacing $s$ with $s+i$, $p$ with $j$, and $t$ with $s+j$, we get
$$W^*_{m,r}[s+i+j,s+j]_{q}=\sum_{k=s}^{min\{s+j,s+i\}}W^*_{m,r}[s+i,k]_{q}W^*_{m,r+mk}[j,s+j-k]_{q}.$$
This gives the following LU factorization of the matrix
\begin{align*}
&
\begin{bmatrix}
 W^*_{m,r}[s,s]_{q} & W^*_{m,r}[s+1,s+1]_{q}  & \ldots  & W^*_{m,r}[s+n,s+n]_{q}  \\
 W^*_{m,r}[s+1,s]_{q} & W^*_{m,r}[s+2,s+1]_{q}  & \ldots  & W^*_{m,r}[s+n+1,s+n]_{q}  \\
 \vdots  &\vdots  &\ldots  &\vdots\\
 W^*_{m,r}[s+n,s]_{q} & W^*_{m,r}[s+n+1,s+1]_{q}  & \ldots  & W^*_{m,r}[s+2n,s+n]_{q}  
\end{bmatrix}\;\;\;\;\;\;\;\;\;\;\;\;\;\;\;\;\;\;\;\;\;\;\;\;\;\;\;\;\;\;\;\;\;\;\;\;\;\;\;\;\;\;\;\;\;\;\;\;\\
&\;\;\;\;\;\;\;\;\;\;\;\;\;\;\;\;\;\;\;\;=
\begin{bmatrix}
 W^*_{m,r}[s,s]_{q} & 0  & \ldots  & 0  \\
 W^*_{m,r}[s+1,s]_{q} & W^*_{m,r}[s+1,s+1]_{q}  & \ldots  & 0  \\
 \vdots  &\vdots  &\ldots  &\vdots\\
 W^*_{m,r}[s+n,s]_{q} & W^*_{m,r}[s+n,s+1]_{q}  & \ldots  & W^*_{m,r}[s+n,s+n]_{q}  
\end{bmatrix}\\
&\;\;\;\;\;\;\;\;\;\;\;\;\;\;\;\;\;\;\;\;\times
\begin{bmatrix}
 W^*_{m,r+ms}[0,0]_{q} & W^*_{m,r+ms}[1,1]_{q}  & \ldots  & W^*_{m,r+ms}[n,n]_{q}  \\
 0 & W^*_{m,r+m(s+1)}[1,0]_{q}  & \ldots  & W^*_{m,r+m(s+1)}[n,n-1]_{q}  \\
 \vdots  &\vdots  &\ldots  &\vdots\\
 0 & 0  & \ldots  & W^*_{m,r+m(s+n)}[n,0]_{q}  
\end{bmatrix}
.
\end{align*}
This implies that
\begin{equation*}
\det\left(W^*_{m,r}[s+i+j,s+j]_{q}\right)_{0\leq i,j\leq n}=\left( \prod_{k=0}^{n} W^*_{m,r}[s+k,s+k]_{q}\right) \left( \prod_{k=0}^{n}W^*_{m,r+m(s+k)}[k,0]_{q}\right). 
\end{equation*}
Using equation \eqref{initial1} and the fact that $W^*_{m,r}[n,n]_{q}=1$, we have the following theorem.

\begin{thm}\label{HT}
 For nonnegative integers $n$ and $k$, the Hankel transform for $W_{m,r}[n,k]_{q}$ is given by
 \begin{equation}\label{eqHT}
\det\left(W^*_{m,r}[s+i+j,s+j]_{q}\right)_{0\leq i,j\leq n}=\prod_{k=0}^{n}[m(s+k)+r]_{q}^k.
\end{equation}
\end{thm}

When $q\to1$, the Hankel transform in Theorem \ref{HT} yields the Hankel transform for the $r$-Whitney numbers of the second kind. More precisely,
$$\det\left(W_{m,r}(s+i+j,s+j)\right)_{0\leq i,j\leq n}=\prod_{k=0}^{n}(m(s+k)+r)^k.$$
This further gives the Hankel transform for the classical Stirling numbers of the second kind when $(m,r)=(1,0)$. That is,
$$\det\left(S(s+i+j,s+j)\right)_{0\leq i,j\leq n}=\prod_{k=0}^{n}(s+k)^k.$$

\bigskip
\noindent{\bf Acknowledgement}. This research has been funded by Cebu Normal University (CNU) and the Commission  on Higher Education - Grants-in-Aid for Research (CHED-GIA).

\end{document}